\documentclass{article}
\usepackage{amsmath,amsthm,amssymb,amsfonts,epsfig}

\newtheorem{thm}{Theorem}

\newtheorem{prop}[thm]{Lemma}

\newtheorem{definition}[thm]{Definition}
\def\Fscr{\mathcal{F}}

\title{A Semi-strong Perfect Digraph Theorem}
\author{Stephan Dominique Andres \qquad Helena Bergold\\
Winfried Hochst\"{a}ttler \qquad \qquad  Johanna Wiehe\\
\normalsize FernUniversit\"{a}t in Hagen, Fakult\"{a}t f\"{u}r 
Mathematik und Informatik\\
\normalsize 58084 Hagen, Germany\\
\normalsize \texttt{\{dominique.andres,helena.bergold,winfried.hochstaettler,}\\\normalsize\texttt{johanna.wiehe\}@fernuni-hagen.de}}

\begin{document}\maketitle

\begin{abstract}
  Reed showed that, if two graphs are $P_4$-isomorphic, then
  either both are perfect or none of them is. In this note we will
  derive an analogous result for perfect digraphs.

{\bf Key words:} dichromatic number, perfect graph, perfect digraph

{\bf MSC 2000:} 05C17, 05C20, 05C15
\end{abstract}

\section{Introduction and Notation}
Perfect digraphs have been introduced by Andres and Hochst\"attler
\cite{perfectdigraphs} as the class of digraphs where the clique
number equals the dichromatic number for every induced subdigraph.
Reed \cite{semistrong} showed that, if two graphs are
$P_4$-isomorphic, then either both are perfect or none of them is. In
this note we will derive an analogous result for perfect digraphs.

We start with some definitions.  For basic terminology we refer to
Bang-Jensen and Gutin~\cite{digraphs}.  For the rest of the paper, we
only consider digraphs without loops.  Let $D=(V,A)$ be a digraph. The
\emph{symmetric part} $S(D)$ of $D=(V,A)$ is the digraph $(V,A_2)$
where $A_2$ is the union of all pairs of antiparallel arcs of $D$, the
\emph{oriented part} $O(D)$ of $D$ is the digraph $(V,A_1)$ where
$A_1=A\setminus A_2$. 

A proper $k$-coloring of $D$ is an assignment $c:V\to \{1,\ldots,k\}$
such that for all $1\le i \le k$ the digraph induced by
$c^{-1}(\{i\})$ is acyclic.  The \emph{dichromatic number} $\chi(D)$
of $D$ is the smallest nonnegative integer $k$ such that $D$ admits a proper $k$-coloring.
A \emph{clique} in a digraph $D$ is a subdigraph in which for any two
distinct vertices $v$ and $w$ both arcs $(v,w)$ and $(w,v)$ exist.
The \emph{clique number} $\omega(D)$ of $D$ is the size of the largest
clique in $S(D)$.  The clique number is an obvious lower bound for the
dichromatic number.  $D$ is called \emph{perfect} if, for any induced
subdigraph $H$ of~$D$, $\chi(H)=\omega(H)$.

An (undirected) graph $G=(V,E)$ can be considered as the symmetric
digraph $D_G=(V,A)$ with $A=\{(v,w),(w,v)\mid vw\in E\}$. In the
following, we will not distinguish between $G$ and $D_G$. In this way,
the dichromatic number of a graph $G$ is its chromatic number
$\chi(G)$, the clique number of $G$ is its usual clique number
$\omega(G)$, and $G$ is perfect as a digraph if and only if $G$ is
perfect as a graph.   

A main result of \cite{perfectdigraphs} is the following:

\begin{thm}[\cite{perfectdigraphs}]\label{main}
A digraph $D=(V,A)$ is perfect if and only if $S(D)$ is perfect and $D$ 
does not 
contain any directed cycle $\vec{C}_n$ with $n\ge3$ as induced subdigraph.
\end{thm}

Together with the Strong Perfect Graph Theorem (see e.g.
\cite{golumbic}) this yields a characterization of perfect digraphs in
form of forbidden induced minors. The Weak Perfect Graph Theorem (see
\cite{golumbic}), though, does not generalize. The directed 4-cycle $\vec{C_4}$ is
not perfect but its complement is perfect, thus perfection is in
general not maintained under taking complements.

Two graphs $G=(V,E_1)$ and $H=(V,E_2)$ are {\em $P_4$-isomorphic}, if
any set $\{a,b,c,d\}\subseteq V$ induces a chordless path, i.e.\ a
$P_4$, in $G$ if and only if it induces a $P_4$ in $H$.

\begin{thm}[Semi-strong Perfect Graph Theorem \cite{semistrong}]\label{reedresult}
  \label{semistrong} If $G$ and $H$ are $P_4$-isomorphic, then 
\[G \text{ is perfect } \Longleftrightarrow H \text{ is perfect}.\]   
\end{thm}

The graphs without an induced $P_4$ are the cographs \cite{corneil}. Thus any
pair of cographs with the same number of vertices is $P_4$-isomorphic.
In order to generalize 
Theorem~\ref{reedresult} to digraphs we consider the
class of directed cographs~\cite{dicographs}, which are characterized
by a set $\Fscr$ of eight forbidden induced minors. Since the class of
directed cographs is invariant under taking complements and perfect
digraphs are not, it is clear that isomorphism with respect to
$\Fscr$ will not yield the right notion of isomorphism for our
purposes. It turns out that restricting to five of these minors
yields the desired result.

\section{$P^4C$-isomorphic digraphs}
The five forbidden induced minors from \cite{dicographs} we need are
the symmetric path $P_4$, the directed
3-cycle $\vec{C_3}$, the directed path $\vec{P_3}$
and the two possible augmentations $\vec{P}_3^+$ and
$\vec{P}_3^-$ of the $\vec{P_3}$ with one antiparallel edge (see
Figure~\ref{fig:1}).

\begin{figure}[htbp]
  \centering
\includegraphics[width=.7\textwidth]{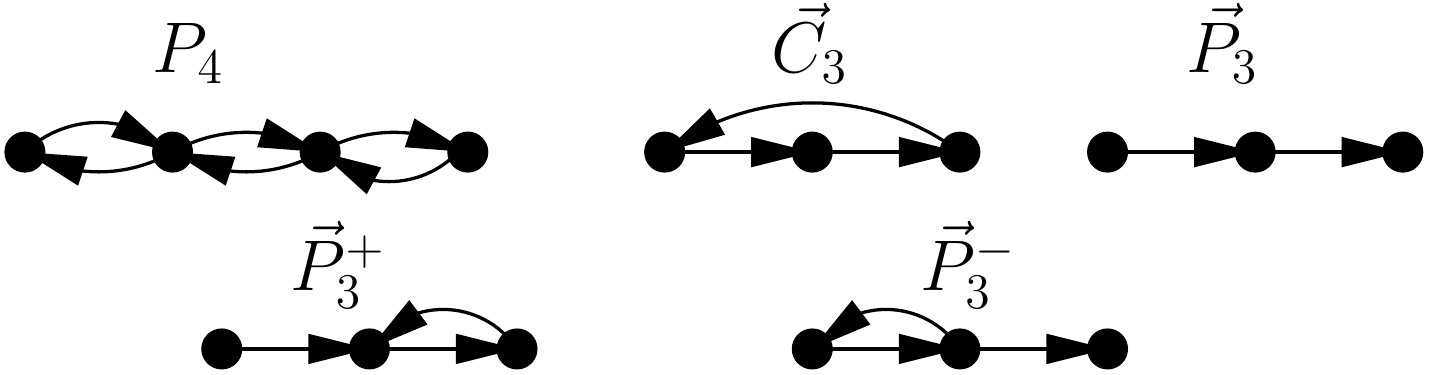}
  \caption{The five induced subdigraphs considered\label{fig:1}}
\end{figure}

\begin{definition}
  Let $D=(V,A)$ and $D'=(V,A')$ be two digraphs on the same vertex
  set. Then $D$ and $D'$ are said to be {\em $P^4C$-isomorphic} if and
  only if
  \begin{enumerate}
  \item any set $\{a,b,c,d\} \subseteq V$ induces
a~$P_4$ in $S(D)$ if and only if it induces
a~$P_4$ in~$S(D')$,
  \item any set $\{a,b,c\} \subseteq V$ induces a $\vec{C_3}$ in $D$
    if and only if it induces a $\vec{C_3}$ in~$D'$,
  \item any set $\{a,b,c\} \subseteq V$ induces a $\vec{P_3}$ with midpoint $b$ in $D$
    if and only if it induces a $\vec{P_3}$ with midpoint $b$ in $D'$ and 
  \item any set $\{a,b,c\} \subseteq V$ induces a $\vec{P}_3^+$ or a $\vec{P}_3^-$ in either case with  midpoint $b$ in $D$
    if and only if it induces one of them with midpoint $b$ in $D'$.
  \end{enumerate}
\end{definition}

Note that the 
$P_4$ in case 1 is not
necessarily induced in $D$, resp.\ in $D'$.

\begin{prop}\label{prop:cycle}
  If $D$ and $D'$ are $P^4C$-isomorphic, then $D$ contains an induced
  directed cycle of length $k\ge 3$ if and only if the same is true for $D'$.
\end{prop}

\begin{proof}
  By symmetry it suffices to prove that, if $\{v_0,\ldots v_{k-1}\}$
  induces a directed cycle $\vec{C_k}$ in $D$, then the same holds for $D'$.
  The assertion is clear if $k=3$, thus assume $k \ge 4$. We may,
  furthermore, assume that the vertices are traversed in consecutive
  order in $D$. Since $D$ and $D'$ are $P^4C$-isomorphic, each set
  $\{v_i,v_{i+1}, v_{i+2}\}$ induces a $\vec{P_3}$ with midpoint
  $v_{i+1}$ in $D'$, where indices are taken modulo $k$. This yields a
  directed cycle $C$ on $v_0,\ldots v_{k-1}$, possibly with opposite
  orientation wrt.~$D$. In that case we relabel the vertices such
  that the label coincides with the direction of traversal. We claim
  the cycle is induced in $D'$, too.

  Assume it is not, i.e.\ $C$ has a chord $(v_i,v_j)$, $j\ne i-1$
  in~$D'$. We choose $j$ such that the directed path from $v_j$ to
  $v_i$ on $C$ is shortest possible. If $(v_i,v_j)$ is an asymmetric
  arc, then, since $\{v_i,v_j,v_{j+1}\}$ does not induce a
  $\vec{C_3}$, it must induce a $\vec{P_3}$ with midpoint $v_j$
  in~$D'$ and hence the same must hold in $D$, contradicting
  $\vec{C_k}$ being induced. If we have a pair of antiparallel edges
  between $v_i$ and $v_j$, then, similarly, $\{v_i,v_j,v_{j+1}\}$
  induces a $\vec{P}_3^+$ or a $\vec{P}_3^-$ with midpoint $v_j$, also
  leading to a contradiction.
\end{proof}

\begin{thm} If $D$ and $D'$ are $P^4C$-isomorphic then
\[D \text{ is perfect } \Longleftrightarrow D' \text{ is perfect}.\]   
\end{thm}

\begin{proof}
  By assumption $S(D)$ and $S(D')$ are $P_4$-isomorphic, hence using
  Theorem~\ref{semistrong} we find that $S(D)$ is perfect if and only
  if $S(D')$ is perfect. By Proposition~\ref{prop:cycle}, $D$ contains
  an induced directed cycle of length at least three if and only if
  the same holds for $D'$. The assertion thus follows from Theorem~\ref{main}.
\end{proof}

\section{Transitive extensions of cographs}
In this section we will analyse the class of digraphs without any of
the five subgraphs, which thus are trivially pairwise
$P^4C$-isomorphic.

Since the symmetric part of such a graph is a cograph, we may consider
its cotree~\cite{corneil} in canonical form, where the labels
alternate between $0$ and $1$. Since the $1$-labeled tree vertices
correspond to complete joins, there is no additional room for
asymmetric arcs. The $0$-labeled vertices correspond to disjoint
unions. Assume the connected components in $S(G)$ are $G_1,\ldots G_k$.
\begin{prop}
  If there exists an asymmetric arc connecting a vertex $v_i$ in $G_i$ to a
  vertex $v_j$ in~$G_j$, then $G_i$ and $G_j$ are connected by an orientation
  of the complete bipartite graph $K_{V(G_i),V(G_j)}$.
\end{prop}
\begin{proof}
  Since $S(G_i)$ and $S(G_j)$ are connected and by symmetry, it
  suffices to show that $v_i$ must be connected by an asymmetric arc to all
  symmetric neighbors of~$v_j$. Let $w$ be such a neighbor. Since
  there is no symmetric arc from $v_i$ to $w$ and $\{v_i,v_j,w\}$ must
  neither induce a $\vec{P}_3^-$ nor a $\vec{P}_3^+$, we must have an
  asymmetric arc between $v_i$ and $w$.
\end{proof}

Hence, the asymmetric arcs between the components $G_1,\ldots,G_k$
constitute an orientation of a complete $\ell$-partite graph for $1\le
\ell \le k$. The situation is further complicated by the fact that we
must neither create a $\vec{C_3}$ nor a $\vec{P_3}$, where we have to
take into account that there may also be asymmetric arcs within the $G_i$.

We wonder whether this structure is strict enough to make some
problems tractable that are $\mathcal{NP}$-complete in general. In
particular we would be interested in the complexity of the problem to cover all vertices with a minimum number of vertex disjoint directed paths.

\end{document}